\documentclass{conm-p-l}

\usepackage{amssymb}

\usepackage{}

\def\bbE{\mathrm{I\!E}}

\def\bbR{\mathrm{I\!R}}
\def\rn{\bbR\nh^n}

\def\dx{\hs\dot{\nh x\hs}\nh}

\def\bp{\mathbf{p}}
\def\dbp{\dot{\hskip0pt\bp\hskip0pt}}
\def\hbp{\hat{\mathbf{p}}}

\def\bx{\mathbf{x}}

\def\by{\mathbf{y}}
\def\dby{\dot{\hskip0pt\by\hskip0pt}}
\def\hby{\hat{\mathbf{y}}}
\def\ddby{\ddot{\hskip0pt\by\hskip0pt}}

\def\df{d\hskip-.8ptf}
\def\cj{c}
\def\p{p}
\def\wp{warp\-ed\hh-\hn prod\-uct}
\def\wp{warp\-ed\hh-\hn prod\-uct}
\def\hg{\hat{g\hskip2pt}\hskip-1.3pt}

\def\hyp{\hskip.5pt\vbox
{\hbox{\vrule width2.5ptheight0.5ptdepth0pt}\vskip2pt}\hskip.5pt}
\def\hs{\hskip.7pt}
\def\hh{\hskip.4pt}
\def\nh{\hskip-.7pt}
\def\nnh{\hskip-1.5pt}
\def\hn{\hskip-.4pt}
\def\w{^{\phantom i}}
\def\bM{\hskip3pt\overline{\hskip-3ptM\nh}\hs}
\def\bna{\hs\overline{\nh\nabla\nh}\hs}
\def\bg{\hskip1.2pt\overline{\hskip-1.2ptg\hskip-.3pt}\hskip.3pt}
\def\bM{\hskip3pt\overline{\hskip-3ptM\nh}\hs}
\def\br{\hskip3pt\overline{\hskip-3ptR\nh}\hs}
\def\vg{\varGamma}
\def\ve{\varepsilon}
\def\h{\eta}
\def\bvg{\hskip3pt\overline{\hskip-3pt\vg\nh}\hs}

\newtheorem{theorem}{Theorem}[section]
\newtheorem{lemma}[theorem]{Lemma}
\newtheorem{corollary}[theorem]{Corollary}

\theoremstyle{definition}

\newtheorem{example}[theorem]{Example}

\theoremstyle{remark}
\newtheorem{remark}[theorem]{Remark}

\numberwithin{equation}{section}

\begin{document}

\title{Max\-i\-mal\-ly-\nh warp\-ed metrics with harmonic curvature}
%\title{}


%    author one information
\author[A. Derdzinski]{Andrzej Derdzinski}
\address{Department of Mathematics\\
The Ohio State University\hskip-1pt\\
231 W\hskip-2pt. 18th Avenue\\
Columbus, OH 43210, USA}
%\curraddr{}
\email{andrzej@math.ohio-state.edu}
\thanks{Both authors' research was supported in part by a 
FAPESP\hn-\hh OSU 2015 Regular Research Award (FAPESP grant: 
2015/50265-6)}

%    author two information
\author[P. Piccione]{Paolo Piccione}
\address{Departamento de Matem\'atica\\
Instituto de Matem\'atica e Estat\'\i stica\\
Universidade de S\~ao Paulo\\
Rua do Mat\~ao 1010, CEP 05508-900\\
S\~ao Paulo, SP, Brazil}
%\curraddr{}
\email{piccione@ime.usp.br}
\thanks{}

%    The 2010 edition of the Mathematics Subject Classification is
%    the current definitive version.
\subjclass[2010]{Primary 53C25; Secondary 53B20}

\date{December 13, 2018}

\begin{abstract}
We describe the local structure of Riemannian manifolds with harmonic 
curvature which admit a maximum number, in a well-defined sense, of local 
warped-\hn product decompositions, and at the same time their Ricci 
tensor has, at some point, only simple eigenvalues. We also prove that, in 
every given dimension greater than two, the local-isometry types of 
such manifolds form a finite\hh-dimensional moduli space, and a 
nonempty open subset of this moduli space is realized by 
locally irreducible complete metrics which are neither Ric\-ci-par\-al\-lel, 
nor -- for dimensions greater than three -- con\-for\-mal\-ly flat.
\end{abstract}

\maketitle


\setcounter{section}{0}
\section*{Introduction}%\label{intro}
\setcounter{equation}{0}
A Riemannian manifold is said to have {\it harmonic curvature\/} 
\cite[Sect.~16.33]{besse} if
\begin{equation}\label{dvr}
\mathrm{div}\,R\,=\,0,\hskip14pt\mathrm{or,\ in\ local\ coordinates,\ 
}\,\,R_{i\hn j\hh l}\w{}\hs^k{}\nnh_{,\hs k}\w=0\hh,
\end{equation}
$R\,$ being the curvature tensor. We consider har\-mon\-ic-cur\-va\-ture 
Riemannian manifolds $\,(M\nh,g)\,$ of dimensions $\,n\ge3\,$ in which, with 
$\,\mathrm{r}\,$ denoting the Ric\-ci tensor,
\begin{equation}\label{dia}
\begin{array}{l}
\mathrm{r}\,\mathrm{\ has\ }\hs n\mathrm{\ distinct\ eigen\-values\ at\ 
some\ point,\ and\ an\ open\ sub\-man\-i}\hyp\\
\mathrm{fold\ of\ }\,\hs(M\nh,g)\hh\,\mathrm{\ admits\ a\ nontrivial\ 
warp\-ed}\hyp\mathrm{prod\-uct\ decomposition.}
\end{array}
\end{equation}
Such \wp\ decompositions have one\hs-di\-men\-sion\-al fibres 
(Corollary~\ref{multp}), and are in a one\hs-to\hh-one correspondence with 
certain one\hh-di\-men\-sion\-al Lie sub\-al\-ge\-bras of 
$\,\mathfrak{isom}(M\hn'\nnh,g')$, the Lie algebra of Kil\-ling fields on the 
Riemannian universal covering $\,(M\hn'\nnh,g')\,$ of $\,(M\nh,g)\,$ (see 
Remark~\ref{liesa}). The number $\,\gamma\,$ of these sub\-al\-ge\-bras cannot 
exceed $\,n\hn-\nh1$, cf.\ Corollary~\ref{gleno}, and we refer to $\,g\,$ as 
{\it max\-i\-mal\-ly warp\-ed\/} if
\begin{equation}\label{gno}
\gamma\,\,=\,\,n\hn-\nh1\hh. 
\end{equation}
Our Theorem~\ref{lcstr} describes the local structure of Riemannian manifolds 
$\,(M\nh,g)$ of dimensions $\,n\ge3$, satisfying (\ref{dvr}) -- (\ref{gno}). 
Their lo\-cal-i\-som\-e\-try types turn out to form a 
$\,(2n-3)$-di\-men\-sion\-al moduli space (Remark~\ref{modsp}), and we prove 
(in Theorem~\ref{cpopn}) that some nonempty open 
subset of the moduli space consists of lo\-cal-i\-som\-e\-try types of such 
manifolds which in addition are
\begin{equation}\label{gnr}
\begin{array}{l}
\mathrm{complete,\ locally\ irreducible,\ and\ neither\ 
con}\hyp\\
\mathrm{for\-mal\-ly\ flat\ (unless\ }\hs n=3\mathrm{),\ nor\ 
Ric\-ci}\hyp\mathrm{par\-al\-lel.}
\end{array}
\end{equation}
We do not know if any compact manifold can have the properties (\ref{dvr}) -- 
(\ref{gno}). However, we observe (see Theorem~\ref{glwpp}) that compact 
Riemannian manifolds with harmonic curvature that admit {\it global\/} 
nontrivial warp\-ed-\hn prod\-uct decompositions must have fibre dimensions 
greater than one and, consequently, cannot be {\it Ric\-ci-ge\-ner\-ic\/} in 
the sense of satisfying the dis\-tinct-eigen\-val\-ues clause of (\ref{dia}).

Harmonicity of the curvature always follows if the metric is {\it 
Ric\-ci-par\-al\-lel}, or {\it locally reducible with 
har\-mon\-ic-cur\-va\-ture factors}, or {\it con\-for\-mal\-ly flat and of 
constant scalar curvature\/} while, in dimension three, harmonic curvature 
amounts to con\-for\-mal flat\-ness plus constancy of the scalar curvature 
\cite[Sect.~16.35 and~16.4]{besse}.

Compact Riemannian manifolds with (\ref{dvr}) have been studied extensively. 
All their known examples (aside from the three classes italicized above), 
listed as 2,3,4 in \cite[p.~432]{besse}, admit -- at least locally -- 
nontrivial warp\-ed-\hn prod\-uct decompositions with a fibre of dimension 
greater than one. Consequently (see Corollary~\ref{multp} below), they are not 
Ric\-ci-ge\-ner\-ic. However, for examples 2 and 4 of \cite[p.~432]{besse} 
the warp\-ed-\hn prod\-uct structure, rather than being an Ansatz, is a 
consequence of geometric conditions involving multiplicities of 
eigen\-val\-ues of the Ric\-ci tensor \cite{derdzinski-82} or self-dual Weyl 
tensor \cite{derdzinski-88}. The following questions about compact Riemannian 
manifolds with harmonic curvature, lying outside of the three italicized 
classes, are thus open and natural: can they be Ric\-ci-ge\-ner\-ic? must 
they, locally, have a warp\-ed-\hn prod\-uct decomposition and if not, how to 
describe those among them which have one?

Our Theorem~\ref{cpopn} yields an affirmative answer to a weaker version of 
the first question in which completeness replaces compactness. On the other 
hand, the second question provides an obvious motivation for studying 
condition (\ref{gno}).

The authors wish to thank the referees for constructive comments, which have 
resulted in a substantial improvement of the presentation.

\section{Preliminaries}\label{pr}
\setcounter{equation}{0}
Manifolds (always assumed connected), mappings and tensor fields are by 
definition $\,C^\infty\nnh$-dif\-fer\-en\-ti\-a\-ble. By a {\it Co\-daz\-zi 
tensor\/} \cite[p.~435]{besse} on a Riemannian manifold one means a 
twice\hh-co\-var\-i\-ant symmetric tensor field $\,S\,$ with a totally 
symmetric covariant derivative $\,\nabla\nnh S$. One then has two well-known 
facts \cite[Sect.~16.4(ii)]{besse}:
\begin{equation}\label{cod}
\begin{array}{rl}
\mathrm{i)}&\mathrm{div}\,R=0\,\,\mathrm{\ if\ and\ only\ if\ 
}\,\,\mathrm{r}\,\,\mathrm{\ is\ a\ Co\-daz\-zi\ tensor,}\\
\mathrm{ii)}&\mathrm{the\ condition\ }\,\mathrm{div}\,R=0\hs\mathrm{\ implies\ 
constancy\ of\ }\,\mathrm{s}\hh,
\end{array}
\end{equation}
$\,\mathrm{s}\,$ being the scalar curvature. As shown by DeTurck and 
Goldschmidt \cite{deturck-goldschmidt},
\begin{equation}\label{ana}
\mathrm{metrics\ with\ }\,\mathrm{div}\,R=0\,\mathrm{\ are\ 
real}\hyp\mathrm{analytic\ in\ suitable\ local\ coordinates.}
\end{equation}
We call two (connected) real\hh-an\-a\-lyt\-ic Riemannian manifolds {\it 
locally isometric\/} if they have open sub\-man\-i\-folds that are both 
isometric to open sub\-man\-i\-folds of a third such manifold. One easily sees 
that this is an equivalence relation. In view of the extension theorem for 
analytic isometries \cite[Corollary~6.4 on p.~256]{kobayashi-nomizu}, for two 
complete real\hh-an\-a\-lyt\-ic Riemannian manifolds, being locally isometric 
to each other means the same as having isometric Riemannian universal 
coverings.

On a manifold with a tor\-sion-free connection $\,\nabla\nh$, the Ric\-ci 
tensor $\,\mathrm{r}\,$ satisfies the Boch\-ner identity 
$\,\mathrm{r}(\,\cdot\,,v)+\hs d\hs[\mathrm{div}\hs v]
=\hs\mathrm{div}\hskip1.2pt\nabla\nh v$, where $\,v\,$ is any vector field. 
Its coordinate form $\,R\nh_{jk}\w v^{\hs k}\nh=v^{\hs k}{}_{,\hs jk}\w\nh
-v^{\hs k}{}_{,\hs kj}\w$ arises via contraction from the Ric\-ci identity 
$\,v^{\hs l}{}_{,\hs jk}\w\nh
-v^{\hs l}{}_{,\hs kj}\w\nh=R\nh_{jkq}\w{}^l\hs v\hh^q\nh$. (We use the sign 
convention for $\,R\,$ such that 
$\,R\nh_{jk}\w\nh=R\nh_{jqk}\w{}^q\nh$.) Applied to the gradient $\,v\,$ of a 
function $\,\phi\,$ on a Riemannian manifold, this yields
\begin{equation}\label{rcd}
\mathrm{r}\hs(\nabla\nnh\phi\nh,\,\cdot\,)\,+\hs\,d\hh\Delta\phi\,
=\,\hs\mathrm{div}\hh[\nabla\nh d\phi]\hh.
\end{equation}
The {\it warped product\/} of Riemannian manifolds $\,(\bM\nh,\bg)\,$ and 
$\,(\varSigma,\h)\,$ with the {\it warping function\/} 
$\,\phi:\bM\to(0,\infty)\,$ is the Riemannian manifold
\begin{equation}\label{war}
(M\nh,g)\,=\,(\bM\times\varSigma,\,\bg+\nh\phi^2\h).
\end{equation}
(The same symbols $\,\bg,\h,\phi\,$ stand here for also the pull\-backs of 
$\,\bg,\h,\phi\,$ to the product $\,M\nh=\bM\times\varSigma$.) One calls 
$\,(\bM\nh,\bg)\,$ and $\,(\varSigma,\h)\,$ the {\it base\/} and {\it 
fibre\/} of (\ref{war}), and refers to (\ref{war}) as {\it nontrivial\/} if 
$\,\phi\,$ is nonconstant. From now on we assume that $\,\dim\varSigma\ge1$.
\begin{remark}\label{wrpro}As 
$\,\bg+\nh\phi^2\h=\phi^2\hh[\hs\phi^{-\nh2}\nh\bg+\h\hs]$, a warped product 
is nothing else than a Riemannian manifold con\-for\-mal to a Riemannian 
product via multiplication by a positive function which is constant along one 
of the factor manifolds.
\end{remark}
A proof of the following well-known lemma \cite{kim-cho-hwang} is given in the 
Appendix.
\begin{lemma}\label{warhc}A warped product\/ {\rm(\ref{war})} with a 
nonconstant warping function\/ $\,\phi$ has harmonic curvature if and only 
if the Le\-vi-Ci\-vi\-ta connection\/ $\,\bna\,$ of $\,(\bM\nh,\bg)$, its 
Ric\-ci tensor $\,\overline{\mathrm{r}}\hn$ and the $\,\bg$-grad\-i\-ent 
$\,\bna\nnh\phi\,$ of $\,\phi\,$ satisfy three conditions\hs{\rm:}
\begin{enumerate}
  \def\theenumi{{\rm\alph{enumi}}}
\item $(\varSigma,\h)\,$ is an Ein\-stein manifold, with some Ein\-stein 
constant\/ $\,\kappa$.
\item $\overline{\mathrm{r}}-\hh\p\phi^{-\nnh1}\bna\nh d\phi\,$ is a 
Co\-daz\-zi\ tensor on\/ $\,(\bM\nh,\bg)$, where\/ $\,\p=\dim\varSigma\ge1$. 
\item $\phi^3\hh\overline{\mathrm{div}}\hh[\phi^{-\nnh1}\bna\nh d\phi]
=[(\p-1)\varLambda-\kappa]\,d\phi+(1-\p)\phi\,d\varLambda/2$, with 
$\,\varLambda=\bg(\bna\nnh\phi\hn,\nnh\bna\nnh\phi)$
\end{enumerate}
and the\/ $\bg$-di\-ver\-gence\/ $\hs\overline{\mathrm{div}}\nh$. \nnh Then, 
at each point of\/ $M\nnh$, for the Ric\-ci tensor\/ $\mathrm{r}\nh$ of\/ $g$,
\begin{enumerate}
  \def\theenumi{{\rm\alph{enumi}}}
\item [(d)]the space tangent to the fibre factor is contained in an 
eigen\-space of\/ $\,\mathrm{r}$.
\end{enumerate}
One may rewrite\/ {\rm(c)} as a requirement involving the\/ 
$\,\bg$-La\-plac\-i\-an\/ $\,\hs\overline{\nh\Delta\nh}\hs\phi\hn$, namely,
\begin{enumerate}
  \def\theenumi{{\rm\alph{enumi}}}
\item [(e)] $\phi^2[\overline{\mathrm{r}}(\bna\nnh\phi\nh,\,\cdot\,)
+d\hh\hs\overline{\nh\Delta\nh}\hs\phi]
=[(\p-1)\varLambda-\kappa]\,d\phi+(1-\p/2)\phi\,d\varLambda$.
\end{enumerate}
Finally, when\/ $\,\p=1$, and so\/ $\,\kappa=0$, {\rm(c)} reads\/ 
$\,\overline{\mathrm{div}}\hh[\phi^{-\nnh1}\bna\nh d\phi]=0$.
\end{lemma}
From (d) and, respectively, (c), we obtain two easy consequences:
\begin{corollary}\label{multp}In a warp\-ed-\hn prod\-uct Riemannian manifold 
with a nonconstant warping function, harmonic curvature, and a fibre of 
dimension greater than one, the Ric\-ci tensor has, at every point, at least 
one multiple eigen\-value. The assumptions\/ {\rm(\ref{dvr})} -- 
{\rm(\ref{dia})} thus imply one-di\-men\-sion\-al\-i\-ty of the fibre for any 
\hbox{warp\-ed-}\hskip0ptprod\-uct decomposition in\/ {\rm(\ref{dia})}.
\end{corollary}
\begin{theorem}\label{glwpp}
If a warp\-ed-\hn prod\-uct Riemannian manifold\/ $\,(M\nh,g)\,$ has a compact 
base\/ $\,(\bM\nh,\bg)$, a nonconstant warping function, and harmonic 
curvature, then the Ein\-stein constant\/ $\,\kappa$ of its fibre must be 
positive. Thus, the dimension of the fibre is greater than one and, in view of 
Corollary\/~{\rm\ref{multp}}, $\,(M\nh,g)\,$ cannot satisfy the 
dis\-tinct-eigen\-val\-ues clause of\/ {\rm(\ref{dia})}.
\end{theorem}
In fact, given a positive function $\,\phi\,$ on a Riemannian manifold 
$\,(\bM\nh,\bg)\,$ and constants $\,\kappa,p\in\bbR$, let us set 
$\,v=\bna\nnh\phi\,$ and $\,w=\phi^2u-[(\p-1)\varLambda-\kappa]\hs v
+(p-2)\hh\phi\hs\bna\hskip-2.7pt_v\w v$, where 
$\,u=\hs\overline{\mathrm{div}}\nh\,\bna\nh v\,$ and 
$\,\varLambda=\bg(v,v)$. Then the function 
$\,\phi\hs^{p-\hn4}\bg(v,w)\,$ differs by a $\bg$-di\-ver\-gence from 
$\,-[(\p-1)\varLambda-\kappa]\phi\hs^{p-\hn4}\nh\varLambda
-\phi\hs^{p-\nh2}\bg(\bna\nh v,\bna\nh v)$, while (c) reads $\,w=0$, as 
$\,2\bna\hskip-2.7pt_v\w v=\bna\nnh\varLambda$. (An easy exercise.)
\begin{remark}\label{ricog}The base and fibre factor distributions of any 
warped product are Ric\-ci-or\-thog\-onal to each other. (See the equality 
$\,R_{ia}\w\nh=0\,$ in formula (\ref{ria}) of the Appendix.) Thus, if the 
base, or fibre, is one\hh-di\-men\-sion\-al, nonzero vectors tangent to it 
constitute eigen\-vectors of the Ric\-ci tensor.
\end{remark}

\section{Vector fields}\label{vf}
\setcounter{equation}{0}
\begin{lemma}\label{cplic}Let a maximal integral curve\/ 
$\,(a_-\w,a_+\w)\ni t\mapsto x(t)\,$ of a vector field\/ $\,v\hs$ on a 
manifold\/ $\,M\nh$, with\/ $\,-\infty\le a_-\w\nnh<a_+\w\nh\le\infty$, some\/ 
$\,t'\nh\in(a_-\w,a_+\w)$, and some compact set\/ $\,C\hh\subseteq M\nh$, 
have the property that\/ $\,x(t)\in C\hs$ for all\/ 
$\,t\in[\hs t'\nh,a_+\w)\,$ or, respectively, for all\/ $\,t\in(a_-\w,t']$. 
Then\/ $\,a_+\w=\infty\,$ or, respectively, $\,a_-\w=-\infty$.

Consequently, a maximal integral curve of a vector field on a manifold, lying 
within a compact set, must be complete, that is, defined on\/ $\,\bbR$.
\end{lemma}
\begin{proof}For a compactly supported function $\,\chi\,$ equal to $\,1\,$ on 
an open set $\,\,U$ containing $\,C\hn$, the curve restricted to 
$\,[\hs t'\nh,a_+\w)$, or to $\,(a_-\w,t']$, clearly remains half-max\-i\-mal 
(not extendible beyond $\,a_+\w$, or $\,a_-\w$) when treated as an integral 
curve of $\,\chi v$. On the other hand, $\,\chi v\,$ is complete due to 
compactness of its support.
\end{proof}
By a {\it section\/} of a locally trivial fibre bundle we mean, as usual, a 
sub\-man\-i\-fold $\,\varSigma\,$ of the total space $\,Q\,$ mapped 
dif\-feo\-mor\-phic\-al\-ly onto the base $\,M\,$ by the bundle projection
$\,\mathrm{p}$. We also identify the section with the inverse 
$\,\psi:M\to\varSigma\,$ of the latter dif\-feo\-mor\-phism, which makes it a 
mapping $\,\psi:M\to Q\,$ having 
$\,\mathrm{p}\circ\nh\psi=\hs\mathrm{Id}_M\w$. In the case of a vector bundle 
$\,Q$, a section $\,\psi$, and a zero $\,z\in M\,$ of $\,\psi$, the 
corresponding sub\-man\-i\-fold $\,\varSigma\,$ of $\,Q\,$ intersects the zero 
section $\,M\,$ at $\,z\,$ (that is, at the zero vector of the fibre 
$\,Q\nh_z\w$), giving rise to the {\it differential\/} 
$\,\partial\hh\psi\nnh_z\w$, defined to be the linear operator 
$\,T\hskip-2.7pt_z\w\nh M\to Q\nh_z\w$ obtained as the composite of the 
ordinary differential of $\,\psi:M\to Q\,$ at $\,z\,$ (the inverse of 
$\,d\mathrm{p}\nh_z\w:T\hskip-2.7pt_z\w\nh\varSigma
\to T\hskip-2.7pt_z\w\nh M$), followed by the di\-rect-sum projection 
$\,T\hskip-2.7pt_z\w\nh Q=T\hskip-2.7pt_z\w\nh M\oplus Q\nh_z\w\to Q\nh_z\w$. 
Relative to any local coordinates at $\,z\,$ and a local trivialization of 
$\,Q$, the components of $\,A=\partial\hh\psi\nnh_z\w$ form the matrix 
$\,[A_j^\lambda]=[\partial\nh_j\w\psi^\lambda]$, with the partial derivatives 
of the components of $\,\psi\,$ evaluated at $\,z$.

Two important examples are provided by zeros $\,z\,$ of $\,\psi=v$, a vector 
field on $\,M\,$ (with $\,Q=T\nnh M$) and of $\,\psi=\df\nnh$, for a function 
$\,f:M\to\bbR\,$ (here $\,Q=T^*\hskip-2ptM$). In the former case, 
$\,A=\partial v\nh_z\w$ (in coordinates: $\,A_j^k=\partial_jv^k$), is the 
infinitesimal generator of the one-pa\-ram\-e\-ter group of linear 
transformations of $\,T\hskip-2.7pt_z\w\nh M\,$ arising as the differentials, 
at the fixed point $\,z$, of the local diffeomorphisms forming the local flow 
of $\,v$. In the latter, $\,\partial\df\nnh_z\w=\hs\mathrm{Hess}_z\w f\nnh$, 
the Hess\-i\-an of $\,f\,$ at the critical point $\,z$.

Let $\,v\,$ be a vector field on a manifold $\,M\nh$, having a zero at 
$\,z\in M\nh$, where one assumes $\,M\,$ either to be an open 
sub\-man\-i\-fold of a vector space $\,Y\nnh$, or to have a sub\-man\-i\-fold 
$\,N\hs$ with $\,z\in N\hs$ such that $\,v\,$ is tangent to $\,N\hs$ at each 
point of $\,N\nh$. In this way $\,v$, or the restriction of $\,v\,$ to 
$\,N\nnh$, becomes a mapping $\,v:M\hn\to\hh Y\nnh$, or a vector field $\,w\,$ 
on $\,N\nnh$. The equality $\,A_j^k=\partial_jv^k$ of the last paragraph, 
evaluated in coordinates for $\,M\,$ which are linear functionals on 
$\,Y\nh$ or, respectively, in which $\,N$ is defined by equating some 
coordinate functions to $\,0$, clearly implies that
\begin{equation}\label{pvz}
\mathrm{i)}\hskip7pt
\partial v\nh_z\w=\hh dv\nh_z\w:Y\nnh\to\hh Y\nh,\hskip13pt
\mathrm{ii)}\hskip7pt
\partial w\nh_z\w\mathrm{\ equals\ the\ restriction\ of\ 
}\,\partial v\nh_z\w\mathrm{\ to\ }\,T\hskip-2.7pt_z\w\nh N\nh.
\end{equation}
\begin{lemma}\label{hsdvf}
Given a zero\/ $\,z\in M\hs$ of a vector field\/ $\,v\hs$ on a manifold\/ 
$\,M\nnh$, with the differential\/ $\,A=\partial v\nh_z\w$, let a function\/ 
$\,f:U\to\bbR\,$ on a neighborhood\/ $\,\,U$ of\/ $\,z$ have\/ 
$\,\df\nnh_z\w=0$. Then\/ $\,d\sigma\nnh_z\w=0\,$ and\/ 
$\,(u,u)\nh_\sigma\w=2(Au,u)\nnh_f\w$ for the directional derivative\/ 
$\,\sigma=d_v\w f:U\to\bbR$, all\/ $\,u\in T\hskip-2.7pt_z\w\nh M\nnh$, the 
Hess\-i\-an\/ $\,(\hskip3pt,\hs)\nnh_f\w=\mathrm{Hess}_z\w f\nnh$, and\/ 
$\,(\hskip3pt,\hs)\nh_\sigma\w=\mathrm{Hess}_z\w\sigma$.
\end{lemma}
\begin{proof}With commas denoting, this time, partial derivatives relative to 
fixed local coordinates on a neighborhood of $\,z$, we have 
$\,\sigma=v\hh^j\nnh f\nnh_{,j}\w$ as well as 
$\,\sigma\nnh_{,\hs k}\w=v\hh^j\nnh f\nnh_{,jk}
+v\hh^j{}\nh_{,k}\w f\nnh_{,j}\w$ and $\,\sigma\nnh_{,\hs kl}\w
=v\hh^j\nnh f\nnh_{,jkl}\w+v\hh^j{}\nh_{,\hh l}\w f\nnh_{,jk}\w
+v\hh^j{}\nh_{,\hs k}\w f\nnh_{,jl}\w+v\hh^j{}\nh_{,\hs kl}\w f\nnh_{,j}\w$. 
At $\,z$, both $\,v\hh^j$ and $\,f\nnh_{,j}\w$ vanish, while 
$\,v\hh^j{}\nh_{,\hs k}\w=A^j_k$. This proves our claim.
\end{proof}
\begin{lemma}\label{trapd}
Let\/ $\,z\in N$ be a zero of a vector field\/ $\,w\hs$ on a manifold\/ 
$\,N$ such that, for some\/ $\,\ve=\pm1$, some Euclidean inner product 
$\,\langle\,,\rangle$ in\/ $\,T\hskip-2.7pt_z\w\nh N\nnh$, and\/ 
$\,A=\ve\partial w\nh_z\w$, the bi\-lin\-e\-ar form\/ 
$\,\langle A\hs\cdot\,,\,\cdot\,\rangle\,$ on\/ 
$\,T\hskip-2.7pt_z\w\nh N\,$ is negative definite. In this case there exist 
arbitrarily small neighborhoods\/ $\,\,U\,$ of\/ $\,z\,$ with the following 
property\/{\rm:} if a maximal integral curve\/ 
$\,(a_-\w,a_+\w)\ni t\mapsto x(t)\,$ of\/ $\,w\hs$ and\/ 
$\,t'\nh\in(a_-\w,a_+\w)\,$ satisfy the condition\/ $\,x(t')\in\,U\nh$, 
where\/ $\,-\infty\le a_-\w\nh<a_+\w\hn\le\infty$, then, denoting by\/ 
$\,\pm\,$ the sign of\/ $\,\ve$, one has\/ $\,a_\pm\w=\pm\infty$, and\/ 
$\,x(t)\in\,U\,$ whenever\/ $\,\ve\hh(t\hn-t')\ge0$.
\end{lemma}
\begin{proof}We fix a Riemannian metric $\,g\,$ on a neighborhood of $\,z\,$ 
in $\,N\hs$ having $\,\langle\,,\rangle=g\nh_z\w$. The required neighborhoods 
$\,\,U\,$ of $\,z\,$ are $\,g$-met\-ric balls centered at $\,z$, small enough 
so as to have compact closures and be dif\-feo\-mor\-phic images, under the 
$\,g$-ex\-po\-nen\-tial mapping at $\,z$, of the corresponding Euclidean balls 
around $\,0$ in $\,T\hskip-2.7pt_z\w\nh N\nnh$. This gives smoothness of 
the function $\,f:U\to\bbR\,$ such that $\,2f\,$ equals 
$\,\mathrm{dist}\hn^2(z,\,\cdot\,)$, the squared $\,g$-dis\-tance from $\,z$, 
and using normal coordinates one obtains 
$\,\mathrm{Hess}_z\w f\nh=\langle\,,\rangle$. If the $\,g$-met\-ric ball 
$\,\,U\,$ is sufficiently small, Lemma~\ref{hsdvf} for $\,v=\ve w\,$ implies 
negativity of $\,\sigma=\ve d_w\w f\,$ on $\,\,U\nh\smallsetminus\{z\}$, as 
$\,\sigma\,$ assumes at $\,z\,$ the critical value $\,0\,$ with a 
neg\-a\-tive-def\-i\-nite Hess\-i\-an. Our claim now easily follows from 
Lemma~\ref{cplic}.
\end{proof}
\begin{remark}\label{mltpl}{\it The same\/} neighborhoods $\,\,U\,$ of $\,z\,$ 
will still satisfy the assertion of Lemma~\ref{trapd} if one replaces $\,w\,$ 
by $\,w/c\,$ for a constant $\,c>0\,$ and 
$\,(a_-\w,a_+\w)\ni t\mapsto x(t)$ by 
$\,(ca_-\w,ca_+\w)\ni t\mapsto x(t/c)$.
\end{remark}
\begin{remark}\label{extkf}For any Kil\-ling field $\,v\,$ on a Riemannian 
manifold $\,(M\nh,g)$, the pair $\,(v,\nh\nabla\nh v)\,$ constitutes a 
parallel section of the vector bundle 
$\,T\nnh M\oplus\,\mathfrak{so}(T\nnh M)$ endowed with a suitable linear 
connection \cite[Remark 17.25 on p.\ 547]{dillen-verstraelen}. Therefore,
\begin{enumerate}
  \def\theenumi{{\rm\roman{enumi}}}
\item a Kil\-ling field on $\,M\,$ is uniquely determined by its restriction 
to any nonempty open subset of $\,M\nh$, while
\item assuming $\,(M\nh,g)\,$ to be simply connected and 
real\hh-an\-a\-lyt\-ic, we conclude that any Kil\-ling field $\,v\,$ on a 
nonempty connected open subset of $\,M\,$ has a unique extension to a 
Kil\-ling field on $\,M\nh$.
\end{enumerate}
Given a nontrivial Kil\-ling vector field $\,v\,$ on a Riemannian manifold and 
a function $\,\theta\nh$, the obvious equality 
$\,\pounds\hskip-1pt_{\theta\nh v}\w g\hs\,
=\,\hs\theta\nh\pounds\hskip-1pt_v\w g\hs
+2\hs d\theta\odot g(v,\,\cdot\,)\,$ clearly implies that
\begin{equation}\label{kil}
\mathrm{if \ }\,\,\hs\theta v\,\,\mathrm{\ is\ also\ a\ Kil\-ling\ field,\ 
}\,\,\theta\,\,\mathrm{\ must\ be\ constant, \ cf.\ Remark~\ref{extkf}(i).}
\end{equation}
\end{remark}

\section{Integrable-complement Kil\-ling fields}\label{ic}
\setcounter{equation}{0}
This section presents a well-known correspondence -- see, for instance, the 
Appendix in \cite{maschler} -- between \wp\ decompositions with a 
one\hh-di\-men\-sion\-al fibre and certain special Kil\-ling fields.

Let $\,v\,$ a nontrivial Kil\-ling field on a Riemannian manifold 
$\,(M\nh,g)\,$ such that, on the dense (by Remark~\ref{extkf}(i)) complement 
of its zero set, the distribution $\,v^\perp$ is integrable. In other words, 
locally, at points with $\,v\ne0$, multiplying $\,v\,$ by a suitable positive 
function one obtains a gradient vector field. Equivalently,
\begin{equation}\label{gvg}
\mathrm{the\ }\,1\hyp\mathrm{form\ }\,g(v,\,\cdot\,)\hs/g(v,v)\mathrm{,\ 
defined\ wherever\ }\,v\ne0\mathrm{,\ is\ closed.}
\end{equation}
Namely, (\ref{gvg}) is necessary: for $\,\xi=\,g(v,\,\cdot\,)$, due to 
skew-sym\-me\-try of $\,\nabla\nh\xi$, the in\-te\-gra\-bil\-i\-ty condition 
$\,\xi\wedge\hs d\hh\xi=0\,$ has the lo\-cal-co\-or\-di\-nate expression 
$\,\xi\hh_{i,j}\w\xi_k\w+\xi\nh_{j,k}\w\xi\hh_i\w+\xi_{k,i}\w\xi_j\w=0$, which 
transvected with $\,v\hh^k$ yields 
$\,v\hh^k\xi_k\w\xi\hh_{i,j}\w=
v\hh^k\xi_{k,j}\w\xi\hh_i\w-v\hh^k\xi_{k,i}\w\xi\nh_j\w$, or
\begin{equation}\label{tps}
2\beta\hh\xi\hh_{i,j}\w\hs=\,\beta\nnh_{,\hh j}\w\xi\hh_i\w\hs
-\,\beta\nnh_{,i}\w\xi\nh_j\w\hh,\hskip12pt\mathrm{where\ }\,\beta
=v\hh^k\xi_k\w=g(v,v).
\end{equation}
Closedness of $\,\xi/\nh\beta\,$ amounts to symmetry of 
$\,\nabla(\xi/\nh\beta)$, and so it now follows since (\ref{tps}) with 
$\,\xi\hh_{i,j}\w\nh=-\hh\xi\nh_{j,i}\w$ implies symmetry of 
$\,\beta^2(\xi\hh_i\w/\nh\beta)\nnh_{,\hh j}\w=\beta\hh\xi\hh_{i,j}\w
-\beta\nnh_{,\hh j}\w\xi\hh_i\w$ in $\,i,j$.

If $\,v\,$ is a Kil\-ling field, $\,g(v,\dx)\,$ is constant along any 
geodesic $\,t\mapsto x=x(t)$, as 
$\,d\hs[g(v,\dx)]/dt
=g(\nabla\hskip-2.7pt_{\hh\dot{\hn x\hh}\hn}\w v,\dx)=0$. Then, with 
the orthogonal complement $\,v^\perp$ only defined away from the zero set of 
$\,v$, one easily sees that
\begin{equation}\label{tgl}
\begin{array}{l}
v\,\,\mathrm{\ is\ orthogonal\ to\ any\ geodesic\ passing\ through\ a\ zero\ 
of\ }\,\,v\mathrm{,\ while}\\
\mathrm{whenever\ (\ref{gvg})\ holds,the\ distribution\ 
}\nh\,v^\perp\nnh\mathrm{\ has\ totally\ geodesic\ leaves.}
\end{array}
\end{equation}
\begin{remark}\label{intco}Local Kil\-ling fields $\,v\,$ satisfying 
(\ref{gvg}), outside of their zero sets, if treated as defined only up to 
multiplication by nonzero constants, stand in a natural one\hh-to\hh-one 
correspondence with local \wp\ decompositions of $\,g$ that have a 
one\hh-di\-men\-sion\-al fibre. Here $\,v\,$ is tangent to the fibre direction.

Namely, such a local decomposition is uniquely determined by the base and 
fibre factor distributions. Just one of them suffices, the other being its 
(necessarily integrable) orthogonal complement. That $\,v\,$ locally spans the 
fibre factor distribution of a warped product follows from Remark~\ref{wrpro} 
and the local version of de Rham's decomposition theorem: in view of 
(\ref{tps}), rewritten as $\,2\beta v\hh^i{}\nnh_{,\hh j}\w\nh
=v\hh^i\beta\nnh_{,\hh j}\w\nh-\beta\hh^{,i}\xi\nh_j\w$, where 
$\,\beta=v\hh^k\xi_k\w=g(v,v)$, and \cite[Theorem 1.159]{besse}, $\,v\,$ is 
$\,\hg$-par\-al\-lel for the con\-for\-mal\-ly related metric 
$\,\hg=g/\nh\beta$, with $\,d_v\w\beta=2g(\nabla\hskip-2.7pt_v\w v,v)=0\,$ due 
to skew-ad\-joint\-ness of $\,\nabla\nh v$. Conversely, for a warped product 
with a one\hh-di\-men\-sion\-al fibre, the required Kil\-ling field $\,v\,$ 
comes from a local flow of local isometries of the fibre (cf.\ formula 
(\ref{gij}) in the Appendix), (\ref{kil}) implying uniqueness of $\,v\,$ up to 
a constant factor.
\end{remark}
\begin{remark}\label{liesa}From Remarks~\ref{extkf} and~\ref{intco} it follows 
that, in the case of a \hbox{real\hh-}\hskip0ptan\-a\-lyt\-ic Riemannian 
manifold $\,(M\nh,g)$, denoting by $\,\mathfrak{isom}(M\hn'\nnh,g')\,$ the Lie 
algebra of Kil\-ling fields on the Riemannian universal covering 
$\,(M\hn'\nnh,g')\,$ of $\,(M\nh,g)$, one has a natural bijective 
correspondence between the one\hh-di\-men\-sion\-al Lie sub\-al\-ge\-bras of 
$\,\mathfrak{isom}(M\hn'\nnh,g')\,$ spanned by Kil\-ling fields $\,v\,$ 
satisfying (\ref{gvg}), and the local \wp\ decompositions, with 
one\hh-di\-men\-sion\-al fibres, of $\,g\,$ restricted to the dense open set 
where $\,v\ne0$. As before, $\,v\,$ is tangent to the fibre direction.
\end{remark}
\begin{lemma}\label{cdone}Let an open ball\/ $\,B$ around\/ $\,0$ in 
a Euclidean\/ $\,n$-space, $\,n\ge2$, admit a connection\/ $\,\nabla$ such 
that all line segments through\/ $\,0$ in\/ $\,B$ are\/ 
$\,\nabla\nnh$-to\-tal\-ly geodesic and tangent at all points\/ 
$\,x\in B\smallsetminus\{0\}\hs$ to some co\-di\-men\-sion-one foliation\/ 
$\,\mathcal{F}\nh$ on\/ $\,B\smallsetminus\{0\}\hs$ having\/ 
$\,\nabla\nnh$-to\-tal\-ly geodesic leaves. Then\/ $\,n=2$.
\end{lemma}
\begin{proof}Fix a leaf $\,L\,$ of $\,\mathcal{F}\nh$ and 
$\,x\in L\,$ such that the $\,\nabla\nnh$-ex\-po\-nen\-tial mapping 
$\,\exp_x\w$ sends a Euclidean open ball\/ $\,B'$ centered at\/ $\,0\,$ in 
$\,T\hskip-2.7pt_x\w\hn B$, dif\-feo\-mor\-phic\-al\-ly, onto a neighborhood\/ 
$\,\exp_x\w(B')\,$ of\/ $\,0\,$ in\/ $\,B$. Thus, 
$\,\exp_x\w(B'\nh\cap T\hskip-2.7pt_x\w\hn L)\smallsetminus J\hs
\subseteq\nh L$, for 
$\,J=B\cap\{qx:q\in(-\infty,0\hs]\}$, and so $\,J\hs\subseteq\nh L$. (The 
leaves of $\,\mathcal{F}\hs$ are locally closed, being, locally, the level 
sets of a submersion.) Hence $\,L\hh\cup\hn\{0\}\,$ is a smooth 
$\,\nabla\nnh$-to\-tal\-ly geodesic sub\-man\-i\-fold of $\,B$, with some 
tangent space $\,V\hn$ at $\,0$, meaning in turn that 
$\,L\hh\cup\hn\{0\}=B\hh\cap\hh V\nh$. Consequently, $\,n=2$, for otherwise 
any two such co\-di\-men\-sion-one sub\-spaces $\,V$ of our Euclidean 
$\,n$-space would have a nontrivial intersection.
\end{proof}
When $\,\mathcal{F}\hs$ is real\hh-an\-a\-lyt\-ic, we can also obtain the 
above assertion by applying, to a sphere $\,\varSigma\,$ around $\,0\,$ in 
$\,B$, Haefliger's theorem \cite{haefliger} which states that a transversally 
orientable real\hh-an\-a\-lyt\-ic co\-di\-men\-sion-one foliation may exist on 
a compact manifold $\varSigma\,$ only if the fundamental group of 
$\,\varSigma\,$ has an element of infinite order.
\begin{remark}\label{zeros}Kobayashi \cite{kobayashi} showed that the zero set 
of any Kil\-ling vector field on a Riemannian manifold $\,(M\nh,g)\,$ is 
either empty, or its connected components are mutually isolated totally 
geodesic sub\-man\-i\-folds of even co\-di\-men\-sions.

For a nontrivial Kil\-ling field $\,v\,$ with (\ref{gvg}), {\it the above 
co\-di\-men\-sions must all equal\/ $\,2$.} This is immediate if one fixes a 
zero $\,z\,$ of $\,v\,$ and applies Lemma~\ref{cdone} to a small ball $\,B\hs$ 
in the normal space at $\,z\,$ of the connected component through $\hs z\hs$ 
such that $\,\exp_z\w$ maps $\,B\hs$ dif\-feo\-mor\-phic\-al\-ly onto a 
sub\-man\-i\-fold $\,N$ of $\,M\nh$, with $\,\nabla$ and 
$\,\mathcal{F}\hs$ denoting the $\,\exp_z\w$-pull\-back of the 
Le\-vi-Ci\-vi\-ta connection of the sub\-man\-i\-fold metric $\,h\,$ on 
$\,N$ and, respectively, of the foliation on $\,N\nh\smallsetminus\{z\}\,$ 
the leaves of which are intersections of $\,N\hs$ and the leaves of 
$\,v^\perp\nnh$, the latter defined wherever $\,v\ne0$. (The local flow of 
$\,v\,$ preserves $\,N\hs$ and $\,h$, and so $\,v\,$ is tangent to $\,N\nh$.) 
The restriction of $\,v\,$ to $\,N\hs$ now constitutes an $\,h$-Kil\-ling 
field $\,w\,$ having just one zero, at $\,z$, and satisfying (\ref{gvg}) (for 
$\,h,w\,$ rather than $\,g,v$), so that (\ref{tgl}) allows us to use 
Lemma~\ref{cdone}.
\end{remark}

\section{Multiply-\nh warp\-ed metrics with $\,\mathrm{div}\hs R=0$}\label{mw}
\setcounter{equation}{0}
\begin{lemma}\label{gamma}
Suppose that the Ric\-ci tensor of a real-an\-a\-lyt\-ic Riemannian\/ 
$\,n\hn$-man\-i\-fold\/ $\,(M\nh,g)\,$ has\/ $\,n\hs$ distinct\ eigen\-values\ 
at\ some\ point and, with the notation of Remark\/~{\rm\ref{liesa}}, 
$\,\mathfrak{a}_2\w,\dots,\mathfrak{a}_m\w$ are distinct 
one\hh-\hskip0ptdi\-men\-sion\-al Lie sub\-al\-ge\-bras of\/ 
$\,\mathfrak{isom}(M\hn'\nnh,g')\,$ spanned by Kil\-ling fields\/ 
$\,v_2\w,\dots,v_m\w$ such that each\/ $\,v=v\nnh_j\w$ satisfies\/ 
{\rm(\ref{gvg})}. Then\/ $\,m\hs\le n$, and\/ $\,g(v\nnh_j\w,v_k\w)=0\,$ as 
well as\/ $\,[v\nnh_j\w,v_k\w]=0\,$ if\/ $\,j\ne k$. Finally, 
$\,g(\nabla\hskip-2.7pt_u\w v\nnh_j\w,v_k\w)=0\,$ whenever\/ 
$\,j,k,l\in\{2,\dots,m\}\,$ and\/ $\,u=v_l\w$.
\end{lemma}
\begin{proof}Remarks~\ref{liesa} and~\ref{ricog} imply that all $\,v\nnh_j\w$, 
wherever nonzero, are mutually nonproportional eigen\-vectors of the Ric\-ci 
tensor, which makes them pointwise orthogonal to one another, as well as 
invariant, up to constant factors -- by (\ref{kil}) -- under each other's 
local flows. Thus, $\,m\le n+1\,$ and, as 
$\,[v,w]=\pounds\hskip-1pt_v\w\hn w\,$ for $\,v=v\nnh_j\w$ and $\,u=v_k\w$, 
one gets $\,[v\nnh_j\w,v_k\w]=cv_k\w$ with some constant $\,c\,$ depending on 
$\,j\,$ and $\,k$. Switching $\,j\,$ and $\,k$, we see that $\,c=0$. Now let 
$\,u=v_l\w$, $\,v=v\nnh_j\w$ and $\,u=v_k\w$, where $\,j,k,l\in\{2,\dots,m\}$. 
We have $\,g(\nabla\hskip-2.7pt_u\w v,w)=0\,$ if $\,u=w\,$ (due to the 
Kil\-ling property of $\,v$) and, therefore, also when $\,v=w\,$ (since 
$\,u,v\,$ commute). Also, $\,g(\nabla\hskip-2.7pt_u\w v,w)=0$ in the remaining 
case, with $\,u,v\,$ different from $\,w\,$ (and hence orthogonal to $\,w$): 
as a consequence of (\ref{tgl}), outside of the zero set of $\,w\,$ the 
distribution $\,w^\perp$ has totally geodesic leaves. This proves the final 
claim of the lemma, implying in turn that, if one had $\,m=n+1$, all 
$\,v\nnh_j$ would be parallel, leading to flatness of $\,g$, and contradicting 
the Ric\-ci-eigen\-values assumption.
\end{proof}

Due to DeTurck and Goldschmidt's real\hh-an\-a\-lyt\-ic\-i\-ty theorem 
(\ref{ana}), we may combine Lemma~\ref{gamma} with Remark~\ref{liesa} and 
Corollary~\ref{multp}, obtaining
\begin{corollary}\label{gleno}Under the assumptions\/ {\rm(\ref{dvr})} -- 
{\rm(\ref{dia})}, the integer\/ $\,\gamma\,$ defined in the Introduction does 
not exceed\/ $\,n\hn-\nh1$.
\end{corollary}

\section{The local structure}\label{ls}
\setcounter{equation}{0}
Given an open interval $\,I\subseteq\bbR$, we introduce a Riemannian metric 
$\,g\,$ on the open set 
$\,I\nnh\times\nnh\bbR\nh^{n\hn-\nh1}\nh\subseteq\rn\nnh$, $\,n\ge2$, by 
declaring its component functions in the Cartesian coordinates 
$\,x^1\nh,x\hh^2\nh,\dots,x\hh^n$ to be
\begin{equation}\label{gjk}
g_{kl}\w\,=\,0\,\,\mathrm{\ if\ }\,\,k\ne l\hh,\hskip12ptg_{11}\w\,
=\,1\hh,\hskip12ptg\nh_{j\hn j}\w=g\nh_{j\hn j}\w(t)\,\,\mathrm{\ for\ }\,\,t
=x^1\,\,\mathrm{\ and\ }\,\,j\ge2\hh,
\end{equation}
where $\,I\ni t
\mapsto(g_{22}\w(t),\dots,g_{nn}\w(t))\in(0,\infty)^{n\hn-\nh1}$ is any 
prescribed smooth curve. We also define the functions 
$\,y_2\w,\dots,y_n\w$ and 
$\,\by\,=\hs\mathrm{diag}\hs(\nh y_2\w,\dots,y_n\w)$ of the variable 
$\,t\in I\nh$, valued in $\,\bbR\,$ and, respectively, in the real vector 
space $\,\bbE\cong\bbR\nh^{n\hn-\nh1}$ of all diagonal 
$\,(n\hn-\nh1)\times(n\hn-\nh1)\,$ matrices, by
\begin{equation}\label{tug}
2y\nh_j\w g\nh_{j\hn j}\w=-\dot g\nh_{j\hn j}\w\,\,\mathrm{\ (no\ summation),\ 
with\ }\,\,(\hskip2.3pt)\hskip-2.2pt\dot{\phantom o}\nh=d/dt\hh.
\end{equation}
\begin{remark}\label{compl}If $\,I\nh=\bbR\,$ while 
$\,\mp y\nh_j(t)\ge\delta\,$ whenever $\,\pm\hh t\,$ is sufficiently large and 
positive, for both signs $\,\pm\hs$, some constant $\,\delta>0$, and all 
$\,j\ge2$, then the above metric $\,g\,$ is complete. In fact, (\ref{tug}) 
gives $\,\log g\nh_{j\hn j}\w(t)\to\infty\,$ as $\,|\hh t|\to\infty$, so 
that $\,g\nh_{j\hn j}\w(t)\ge a\,$ with some constant $\,a\in(0,1]\,$ and all 
$\,t\in\bbR$, which in turn gives $\,g\ge ag'$ (positive 
sem\-i\-def\-i\-nite\-ness of $\,g-ag'$) for the standard Euclidean 
metric $\,g'\nh$. Completeness of $\,g'$ now implies that of $\,g$, as 
$\,g\hh$-bound\-ed sets have compact closures due to the resulting inequality 
$\,\mathrm{dist}\hs\ge\hs a\,\mathrm{dist}\hn'$ between distance functions.
\end{remark}
Let us consider the sec\-ond-or\-der autonomous ordinary differential equation
\begin{equation}\label{hcu}
\ddby\,-\,(\mathrm{tr}\hskip1.7pt\by+\by)\hs\dby\,
=\,(\mathrm{tr}\hskip1.7pt\by^2)\hh\by\hs
-\hs(\mathrm{tr}\hskip1.7pt\by)\hh\by^2
\end{equation}
imposed on a $\,C^2\nh$ curve $\,I\ni t\mapsto\by\in\bbE$, in which 
$\,\by\hskip-.2pt\dby\,$ and $\,\by^2\nh=\by\hskip-.2pt\by\,$ represent 
di\-ag\-o\-nal-ma\-trix products, while $\,\mathrm{tr}\hskip1.7pt\by\,$ also 
denotes $\,\mathrm{tr}\hskip1.7pt\by\,$ times the identity.
\begin{lemma}\label{xmple}For a metric\/ $\,g\,$ on\/ 
$\,I\nh\times\bbR\nh^{n\hn-\nh1}\nh$ defined by\/ {\rm(\ref{gjk})} and the 
corresponding curve\/ 
$\,I\ni t\mapsto\by\,=\hs\mathrm{diag}\hs(\nh y_2\w,\dots,y_n\w)\,$ with\/ 
{\rm(\ref{tug})}, at every point\/ 
$\,(t,\bx)\in I\nh\times\bbR\nh^{n\hn-\nh1}\nnh$, each of 
the coordinate vectors\/ $\,\partial\nh_k\w$, $\,k=1,\dots,n$, is an 
eigen\-vector of the Ric\-ci tensor %\/ $\,\mathrm{r}\hs$ 
of\/ $\,g\,$ with an eigen\-value\/ $\,\mu_k\w$ depending on\/ $\,t$.
\begin{enumerate}
  \def\theenumi{{\rm\alph{enumi}}}
\item Specifically, 
$\,\mu_1\w=\mathrm{tr}\hskip1.7pt\dby-\mathrm{tr}\hskip1.7pt\by^2$ and\/ 
$\,\mu_j\w=\dot y\nh_j\w-y\nh_j\w\hs\mathrm{tr}\hskip1.7pt\by\,$ if\/ 
$\,j\ge2$.
\item The scalar curvature\/ $\,\mathrm{s}\hs$ of\/ $\,g\,$ equals\/ 
$\,2\hskip1.7pt\mathrm{tr}\hskip1.7pt\dby-\mathrm{tr}\hskip1.7pt\by^2\nh
-(\mathrm{tr}\hskip1.7pt\by)^2\nh$.
\item $\partial\nh_2\w,\dots,\partial\nh_n\w$ are\/ $\,g\hn$-Kil\-ling fields 
with integrable orthogonal complements.
\item Given any fixed\/ $\,\bx\in\bbR\nh^{n\hn-\nh1}\nnh$, the curve\/ 
$\,I\nh\ni t\mapsto(t,\bx)\,$ is a\/ $\,g\hn$-ge\-o\-des\-ic.
\item $g\,$ has harmonic curvature if and only if\/ {\rm(\ref{hcu})} holds.
\end{enumerate}
\end{lemma}
\begin{proof}We assume $\,j,k,l\,$ to range over $\,\{2,\dots,n\}\,$ and be 
mutually distinct. Repeated indices are {\it not\/} summed over. First, (c) is 
obvious as $\,g_{11}\w,g_{1\hn j}\w,g\nh_{j\hn j}\w,g\nh_{jk}\w$ only depend 
on $\,t=x^1\nh$. Also, 
$\,\vg_{\hskip-2.7pt11}^{\hs1}=\vg_{\hskip-2.7pt11}^{\hs j}=0$, proving (d), 
while $\,\vg_{\hskip-2.7pt1j}^{\hs1}=\vg_{\hskip-2.7pt1j}^{\hs k}
=\vg_{\hskip-2.7ptj\hn j}^{\hs j}=\vg_{\hskip-2.7ptj\hn j}^{\hs k}
=\vg_{\hskip-2.7ptjk}^{\hs j}=\vg_{\hskip-2.7ptjk}^{\hs l}=0\,$ and 
$\,g^{j\hn j}\vg_{\hskip-2.7ptj\hn j}^{\hs1}=-\vg_{\hskip-2.7pt1j}^{\hs j}
=y\nh_j\w$. Hence $\,R_{11}\w=\mu_1\w$ and $\,g^{j\hn j}R\hn_{j\hn j}\w
=\mu_j\w$ for $\,\mu_1\w,\mu_j\w$ as in (a). This yields (a), and hence (b). 
(Each $\,\partial\nh_k\w$ spans the fibre direction of a \wp\ decomposition, 
and we may use Remark~\ref{ricog}.) Next, 
$\,R_{11,\hs j}\w=R_{1j,\hs1}\w=R_{1j,\hs k}\w=R\hn_{jk,\hs1}\w
=R\hn_{jk,\hs j}\w=R\hn_{j\hn j,\hs k}\w=R\hn_{jk,\hs l}\w=0$. Finally, 
$\,g^{j\hn j}R\hn_{j1,\hs j}\w=y\nh_j\w(\mu_j\w\nh-\hn\mu_1\w)\,$ and 
$\,g^{j\hn j}R\hn_{j\hn j,\hs1}\w=\dot\mu_j\w$, so that (\ref{cod}.i) 
implies (e),
\end{proof}
We refer to a solution $\,I\ni t\mapsto\by\in\bbE\,$ of (\ref{hcu}) as {\it 
maximal\/} if it cannot be extended to a larger open interval, and call it 
{\it Ric\-ci-ge\-ner\-ic\/} whenever the $\,n\,$ values 
$\,\mu_k\w=\mu_k\w(t)\,$ of Lemma~\ref{xmple}(a) are all distinct at some 
$\,t\in I\hs$ (or, equivalently, no two among the functions 
$\,\mu_1\w,\dots,\mu_n\w$ coincide everywhere in $\,I$).
\begin{example}\label{thsol}Two non-Ric\-ci-ge\-ner\-ic maximal solutions of 
(\ref{hcu}) are defined by $\,\by=-\nh2\hh\tanh\hs nt\,$ and 
$\,\by=2\hh\tan\hs nt\,$ (times the identity $\,\mathbf1$), with 
$\,I\nh=\bbR\,$ or $\,I\nh=(-\pi/(2n),\pi/(2n))$. In fact, 
$\,2\dby=n\hh(\by^2\nnh\mp4)\,$ and so $\,\ddby=n\by\dby$, 
while for multiples $\,\by$ of $\,\mathbf1\,$ the right-hand side of 
(\ref{hcu}) vanishes and $\,\mathrm{tr}\hskip1.7pt\by+\by=n\by$.
\end{example}
\begin{example}\label{trext}Any solution 
$\,\by\,=\hs\mathrm{diag}\hs(\nh y_2\w,\dots,y_n\w)\,$ of (\ref{hcu}), where 
$\,n\ge2$, can be {\it trivially extended\/} to the solution 
$\,\mathrm{diag}\hs(\nh y_2\w,\dots,y_n\w,0,\dots,0)\,$ with a number 
$\,m>0$ of additional zero components. The new metric defined using 
(\ref{gjk}) -- (\ref{tug}) is iso\-met\-ric to the Riemannian product of the 
original $\,g\,$ and a flat metric on $\,\bbR\nh^m\nh$.
\end{example}
The set of maximal solutions of (\ref{hcu}) is obviously preserved by the 
group $\,K$ acting on it via replacement of $\,\by\hs$ with 
$\,t\mapsto\pm\by(b\pm t)$, where $\,b\in\bbR\,$ and $\,\pm$ is either sign, 
combined with permutations of the components $\,y_2\w,\dots,y_n\w$. We will 
use the term $\,K\hskip-2pt${\it-e\-quiv\-a\-lence\/} when two maximal 
solutions lie in the same $\,K\hskip-1pt$-\hh or\-bit.
\begin{remark}\label{group}Nonzero real numbers $\,a\hs$ act on maximal 
solutions $\,t\mapsto\by(t)\,$ of (\ref{hcu}) by sending them to 
$\,t\mapsto a\by(at)$. (The new metric arising via (\ref{gjk}) -- (\ref{tug}) 
is iso\-met\-ric to $\,g/a^2\nh$.) The group $\,K\hs$ defined above, obviously 
iso\-mor\-phic to the direct product of the isometry group of $\,\bbR\,$ and 
the symmetric group $\,S\nh_{n\hn-\nh1}\w$, along with the 
mul\-ti\-pli\-ca\-tive group $\,\bbR\smallsetminus\{0\}\,$ acting as described 
here, together generate an action of a sem\-i\-di\-rect product of $\,K\hs$ 
and $\,(0,\infty)$.
\end{remark}
\begin{theorem}\label{lcstr}
For any\/ $\,n\ge3$, the construction summarized by\/ {\rm(\ref{gjk})} -- 
{\rm(\ref{tug})} provides a bijective correspondence between two sets 
consisting, respectively, of 
\begin{enumerate}
  \def\theenumi{{\rm\roman{enumi}}}
\item all\/ $\,K\hskip-2pt$-e\-quiv\-a\-lence classes of maximal 
Ric\-ci-ge\-ner\-ic solutions to\/ {\rm(\ref{hcu})}, and
\item all lo\-cal-i\-som\-e\-try types of\/ Riemannian\/ 
$\,n\hn$-man\-i\-folds with\/ {\rm(\ref{dvr})} -- {\rm(\ref{gno})}.
\end{enumerate}
For the meaning of lo\-cal-i\-som\-e\-try types, see\/ {\rm(\ref{ana})} and 
the paragraph following it.
\end{theorem}
\begin{proof}We need to show that the mapping from (i) to (ii) is: (A) 
well-de\-fin\-ed, (B) injective, and (C) surjective.

Part (A) easily follows from Lemma~\ref{xmple} combined with the comment on 
$\,g/a^2$ in Remark~\ref{group}, the latter applied to $\,a=\pm1$. To obtain 
(B), note that the lo\-cal-isom\-e\-try type of a metric $\,g\,$ arising 
from (\ref{gjk}) -- (\ref{hcu}) determines the $\,K\nnh$-e\-quiv\-a\-lence 
class of the maximal Ric\-ci-ge\-ner\-ic solution $\,t\mapsto\by\,$ of 
(\ref{hcu}). Namely, the $\,g\hh$-Kil\-ling fields 
$\,\partial\nh_2\w,\dots,\partial\nh_n\w$, valued in eigen\-vectors of the 
Ric\-ci tensor of $\,g\,$ (see Lemma~\ref{xmple}), are -- due to the 
Ric\-ci-ge\-ner\-ic condition and (\ref{kil}) -- unique up to permutations and 
multiplication by nonzero constants, which makes $\,y_2\w,\dots,y_n\w$, 
defined by (\ref{gjk}) with 
$\,g\nh_{j\hn j}\w\nh=g(\partial\nh_j\w,\partial\nh_j\w)$, also unique up to 
permutations. The variable $\,t$, being an arc-length parameter of 
$\,g\hh$-ge\-o\-des\-ics orthogonal to 
$\,\partial\nh_2\w,\dots,\partial\nh_n\w$, cf.\ Lemma~\ref{xmple}(d) and 
(\ref{gjk}), is in turn unique up to substitutions by $\,b\pm t$, for 
constants $\,b$, as required.

Finally, to prove (C), we fix $\,(M\nh,g)\,$ of dimension $\,n\ge3\,$ 
satisfying  (\ref{dvr}) -- (\ref{gno}). Corollary~\ref{multp} and (\ref{ana}), 
along with Remarks~\ref{liesa} and~\ref{extkf}(ii), allow us to choose 
$\,\mathfrak{a}_2\w,\dots,\mathfrak{a}_n\w$ and $\,v_2\w,\dots,v_n\w$ as in 
Lemma~\ref{gamma} for $\,m=n$, and a point $\,x\in M$ at which all 
$\,v\nnh_j\w$ are nonzero. (From now on $\,j\,$ ranges over 
$\,\{2,\dots,n\}$.) By the Lie\hh-brack\-et assertion of Lemma~\ref{gamma}, 
the local flow of each $\,v\nnh_j\w$ preserves all $\,v\nnh_j\w$ and, 
consequently, also a unit vector field $\,v_1\w$ on a neighborhood of $\,x$, 
orthogonal to all $\,v\nnh_j\w$. Since $\,v_1\w$ and all $\,v\nnh_j\w$ commute 
with one another, they constitute the coordinate vector fields of a local 
coordinate system $\,x^1\nnh=t,x\hh^2\nh,\dots,x\hh^n$ on a neighborhood of 
$\,x$, in which the metric $\,g\,$ has the form (\ref{gjk}) as a consequence 
of the last two lines of Lemma~\ref{gamma}, with $\,m=n$. (In particular, the 
assertion $\,g(\nabla\hskip-2.7pt_u\w v\nnh_j\w,v_k\w)=0$, for $\,u=v_l\w$ and 
$\,j,k,l\in\{2,\dots,n\}$, applied to $\,j=k$, shows that 
$\,g\nh_{j\hn j}\w\nh=g(v\nnh_j\w,v\nnh_j\w)\,$ only depend on the variable 
$\,t=x^1\nh$.) Now Lemma~\ref{xmple}(e) yields (C).
\end{proof}
\begin{remark}\label{modsp}The component version of (\ref{hcu}) states that 
$\,\ddot y\nh_j\w\nh-(\mathrm{tr}\hskip1.7pt\by+y\nh_j\w)\hs\dot y\nh_j\w$ 
equals $\,y\nh_j\w[\mathrm{tr}\hskip1.7pt\by^2\nh
-(\mathrm{tr}\hskip1.7pt\by)\hh y\nh_j\w]$. A solution $\,t\mapsto\by\,$ of 
(\ref{hcu}) for $\,n\ge3$, with {\it any\/} prescribed value at $\,t=0$, may 
be chosen so as to make the values $\,\mu_1\w(0),\dots,\mu_n\w(0)\,$ mutually 
distinct. (By Lemma~\ref{xmple}(a), this amounts to using $\,\dby(0)\,$ that 
realizes $\hs(\mu_2\w(0),\dots,\mu_n\w(0))$ lying outside a finite union of 
specific hyperplanes in $\,\bbE$.) Consequently, the lo\-cal-i\-som\-e\-try 
types in Theorem~\ref{lcstr}(ii) form a {\it moduli space} of dimension 
$\,2n-3$.
\end{remark}

\section{The scalar-curvature integral}\label{sc}
\setcounter{equation}{0}
Not surprisingly, in the light of (\ref{cod}.ii) and parts (b), (e) of 
Lemma~\ref{xmple},
\begin{equation}\label{int}
\mathrm{s}\,\,=\,\hs2\hskip1.7pt\mathrm{tr}\hskip1.7pt\dby
-\hs\mathrm{tr}\hskip1.7pt\by^2\nh-(\mathrm{tr}\hskip1.7pt\by)^2\,\mathrm{\ 
is\ constant\ whenever\ }\,t\mapsto\by\,\mathrm{\ satisfies\ (\ref{hcu}).}
\end{equation}
\begin{lemma}\label{negsc}For any solution\/ $\,I\ni t\mapsto\by\in\bbE\,$ 
of\/ {\rm(\ref{hcu})} defined on\/ $\,\bbR$, and not identically equal to 
zero, one must have $\,\mathrm{s}\hs<0\,$ in\/ {\rm(\ref{int})}.
\end{lemma}
\begin{proof}Under the assumption that $\,\mathrm{s}\ge0$, (\ref{int}) gives  
$\,2\hskip1.7pt\mathrm{tr}\hskip1.7pt\dby\ge\hs\mathrm{tr}\hskip1.7pt\by^2\nh
+(\mathrm{tr}\hskip1.7pt\by)^2$ for our solution 
$\,\bbR\ni t\mapsto\by\in\bbE$, and so $\,\mathrm{tr}\hskip1.7pt\by\,$ is 
nondecreasing and nonconstant. Fixing $\,t'\nh\in\bbR\,$ such that 
$\,\mathrm{tr}\hskip1.7pt\by(t')\ne0$, we define a constant $\,\cj>0$ by 
$\,(n\hn-\nh1)\hs\cj^2\nh=[\mathrm{tr}\hskip1.7pt\by(t')]^2\nh$. Depending on 
whether $\,\mathrm{tr}\hskip1.7pt\by(t')\,$ is positive or negative, 
monotonicity of $\,\mathrm{tr}\hskip1.7pt\by\,$ gives 
$\,(\mathrm{tr}\hskip1.7pt\by)^2\nh\ge(n\hn-\nh1)\hs\cj^2$ on 
$\,[\hs t',\infty)\,$ or, respectively, on $\,(-\infty,t']$. The Schwarz 
inequality $\,(\mathrm{tr}\hskip1.7pt\bx)^2\nh
\le(n\hn-\nh1)\hskip1.7pt\mathrm{tr}\hskip1.7pt\bx^2$ now shows that 
$\,\mathrm{tr}\hskip1.7pt\by^2\nh\ge\hs\cj^2$ on $\,[\hs t',\infty)$, or on 
$\,(-\infty,t']$. The relation 
$\,2\hskip1.7pt\mathrm{tr}\hskip1.7pt\dby\ge\hs\mathrm{tr}\hskip1.7pt\by^2\nh
+(\mathrm{tr}\hskip1.7pt\by)^2$ (see above) thus yields 
$\,2\hskip1.7pt\mathrm{tr}\hskip1.7pt\dby\ge\hs\cj^2\nh
+(\mathrm{tr}\hskip1.7pt\by)^2\nh$, that is, 
$\,\hh\dot{\hn\alpha\hh}\hn\ge\hs\cj^2$ on $\,[\hs t',\infty)\,$ or 
$\,(-\infty,t']$, where 
$\,\alpha=2\tan^{-\nh1}(\mathrm{tr}\hskip1.7pt\by\nh/\nh\cj)$. Consequently, 
$\,\alpha\to\pm\infty$ as $\,t\to\pm\infty\,$ for some sign $\,\pm\hs$, 
contrary to boundedness of $\,\alpha$.
\end{proof}
\begin{remark}\label{cplor}
A Riemannian manifold $\,(I\nh\times\bbR\nh^{n\hn-\nh1}\nnh,g)\,$ arising from 
(\ref{gjk}) -- (\ref{hcu}), which makes it real\hh-an\-a\-lyt\-ic, may be 
locally isometric to a compact (and hence complete) real\hh-an\-a\-lyt\-ic 
Riemannian manifold, in the sense of the paragraph following (\ref{ana}), even 
if the solution $\,I\ni t\mapsto\by\in\bbE\,$ of (\ref{hcu}) has no extension 
to one defined on $\,\bbR$. This is illustrated by the trivial extension 
(Example~\ref{trext}), with $\hs m>0\hs$ additional zeros, of the solution 
$\,y_2\w(t)=2\hh\tan\hs 2t\,$ of Example~\ref{thsol}, for $\,n=2$, further 
modified using $\,a=1/2\,$ in Remark~\ref{group}, so as to become 
$\,t\mapsto(\tan\hs t,0,\dots,0)$. Since the latter realizes (\ref{tug}) with 
$\,g_{22}\w=\cos^2\nnh t$, it represents, locally, a product of the standard 
sphere $\,\mathrm{S}^2$ with a flat torus $\,\mathrm{T}^m\nh$.
\end{remark}

\section{Completeness}\label{co}
\setcounter{equation}{0}
Let $\,n\ge3$. In the usual fashion, (\ref{hcu}) is equivalent to the 
first-or\-der system
\begin{equation}\label{fos}
\dby\,=\,\bp\hs,\hskip14pt
\dbp\,=\,(\mathrm{tr}\hskip1.7pt\by+\by)\hs\bp\,
+\,(\mathrm{tr}\hskip1.7pt\by^2)\hh\by\hs
-\hs(\mathrm{tr}\hskip1.7pt\by)\hh\by^2.
\end{equation}
Solutions $\,t\mapsto\by\,$ of (\ref{hcu}) thus correspond to integral curves 
$\,t\mapsto(\by,\bp)\,$ of the vector field $\,v\,$ on 
$\,\bbE\hn\times\nh\bbE\,$ represented by (\ref{fos}), and expressed as
\begin{equation}\label{vyz}
(\by,\bp)\,\mapsto\,v_{(\by\nnh,\hh\bp)}\w\hs
=\,(\bp\hh,(\mathrm{tr}\hskip1.7pt\by\nh+\nh\by)\bp
+(\mathrm{tr}\hskip1.7pt\by^2)\by-(\mathrm{tr}\hskip1.7pt\by)\by^2)
\end{equation}
when identified with a mapping 
$\,\bbE\hn\times\nh\bbE\to\bbE\hn\times\nh\bbE$. This $\,v\,$ has an obvious 
curve $\,\bbR\ni q\mapsto q(\mathbf1,\mathbf0)\,$ of zeros, where 
$\,\mathbf1\in\bbE\,$ is the identity. Evaluating the differentials of 
$\,v:\bbE\hn\times\nh\bbE\to\bbE\hn\times\nh\bbE\,$ at 
$\,q(\mathbf1,\mathbf0)$, and of the function 
$\,\bbE\hn\times\nh\bbE\ni(\by,\bp)\mapsto\mathrm{s}
=2\hskip1.7pt\mathrm{tr}\hskip1.7pt\bp-\hs\mathrm{tr}\hskip1.7pt\by^2\nh
-(\mathrm{tr}\hskip1.7pt\by)^2\nh\in\bbR$, cf.\ 
(\ref{int}), at any $\,(\by,\bp)\in\bbE\hn\times\nh\bbE$, we obtain 
$\,dv\nh_{q(\mathbf1\hn,\hh\mathbf0)}\w(\hby,\hbp)=(\hbp,\,nq\hs\hbp
+q^2\hs\mathrm{tr}\hskip1.7pt\hby-(n\hn-\nh1)\hh q^2\hby)\,$ and 
$\,d\hh\mathrm{s}_{(\by\nnh,\hh\bp)}\w(\hby,\hbp)
=2\hh[\hh\mathrm{tr}\hskip1.7pt\hbp-\hs\mathrm{tr}\hskip1.7pt\by\hby
-(\mathrm{tr}\hskip1.7pt\by)\mathrm{tr}\hskip1.7pt\hby]$. When $\,q\ne0$, the 
linear en\-do\-mor\-phism 
$\,dv\nh_{q(\mathbf1\hn,\hh\mathbf0)}\w$ of $\,\bbE\hn\times\nh\bbE\,$ is 
di\-ag\-o\-nal\-iz\-able, with the eigen\-values 
$\,0,n\hh q,(n\hn-\nh1)\hh q,q\,$ of multiplicities 
$\,1,1,n\hn-\nh2,n\hn-\nh2$, the eigen\-space for each of the four 
eigen\-values $\,\lambda\,$ consisting of all $\,(\hby,\hbp)\,$ such that 
$\,\hbp=\lambda\hby\,$ and either $\,\hby$ equals a multiple of the identity 
(for $\,\lambda\in\{0,n\hh q\}$), or $\,\mathrm{tr}\hskip1.7pt\hby=0\,$ (if 
$\,\lambda\in\{(n\hn-\nh1)\hh q,q\}$).

On the other hand, $\,\mathrm{s}\,$ has no critical points in 
$\,\bbE\times\nh\bbE$, and $\,v\,$ is tangent to the level sets of 
$\,\mathrm{s}$. The latter sets are co\-di\-men\-sion-one 
real\hh-an\-a\-lyt\-ic sub\-man\-i\-folds of $\,\bbE\hn\times\nh\bbE$, and 
those among them intersecting the curve 
$\,\bbR\ni q\mapsto q(\mathbf1,\mathbf0)$ correspond, by (\ref{int}), to 
$\,\mathrm{s}=-n(n\hn-\nh1)\hh q^2\nh$, that is, to all nonpositive values of 
$\,\mathrm{s}$. If we fix $\,q\ne0$, the tangent space at 
$\,z=q(\mathbf1,\mathbf0)\,$ of the hyper\-sur\-face $\,N\hs$ given by 
$\,\mathrm{s}=-n(n\hn-\nh1)\hh q^2\nh$, equal to the kernel of 
$\,d\hh\mathrm{s}_{q(\mathbf1\hn,\hh\mathbf0)}\w$, coincides, due to 
dimensional reasons, with the span of the eigen\-spaces of 
$\,dv\nh_{q(\mathbf1\hn,\hh\mathbf0)}\w$ for the three nonzero eigen\-values 
$\,n\hh q,(n\hn-\nh1)\hh q,q$. (See the preceding paragraph and the above 
formula for $\,d\hh\mathrm{s}_{(\by\nnh,\hh\bp)}\w(\hby,\hbp)$.) From 
(\ref{pvz}) it now follows that $\,\partial w\nh_z\w$, for the vector field 
$\,w\,$ on $\,N\hs$ arising as the restriction of $\,v$, is 
di\-ag\-o\-nal\-iz\-able, with positive (or, negative) eigen\-values. 
Thus, as $\,z=q(\mathbf1,\mathbf0)$,
\begin{equation}\label{shl}
\mathrm{our\ }\,z,N\nh,w\,\mathrm{\ and\ }\,\ve=-\mathrm{sgn}\,q\,\mathrm{\ 
satisfy\ the\ hypothesis\ of\ Lemma~\ref{trapd}.}
\end{equation}
\begin{remark}\label{mltdf}Whenever $\,c\in\bbR\smallsetminus\{0\}$, the 
assignment $\,(\by,\bp)\mapsto(c\by,c^2\nh\bp)\,$ is a dif\-feo\-mor\-phism 
$\,F\hskip-3pt_c\w:\bbE\hn\times\nh\bbE\to\bbE\hn\times\nh\bbE$, sending our 
vector field $\,v\,$ to $\,v/c$, and pulling the function $\,\mathrm{s}\,$ 
back to $\,c^2\mathrm{s}$. Using our $\,N\hs$ given by 
$\,\mathrm{s}=-n(n\hn-\nh1)\hh q^2$ we obtain a dif\-feo\-mor\-phism 
$\,(0,\infty)\times\nh N\ni(c,x)\mapsto F(c,x)=F\hskip-3pt_c\w(x)\,$ onto 
the open set in $\,\bbE\hn\times\nh\bbE$ on which $\,\mathrm{s}<0$, as one 
sees defining its inverse by $\,F\hh^{-\nnh1}\nh(x')=(c,x)$, if 
$\,\mathrm{s}(x')<0$, with $\,\hs c,x\,$ such that 
$\,n(n\hn-\nh1)\hh(cq)^2\nh=-\mathrm{s}(x')\,$ and $\,x=F(1/c,x')$.
\end{remark}
In the next theorem, we fix an integer $\,n\ge3$, again denoting by $\,\bbE\,$ 
the space of all diagonal $\,(n\hn-\nh1)\times(n\hn-\nh1)\,$ matrices, and by 
$\,\mathbf1\in\bbE\,$ the identity.
\begin{theorem}\label{cplex}For any\/ $\,(\xi,\zeta)\in\bbR\times(0,\infty)$, 
every maximal solution\/ $\,t\mapsto\by\,$ of\/ {\rm(\ref{hcu})} with\/ 
$\,(\by(0),\dby(0))\,$ sufficiently close to\/ 
$\,(\xi\mathbf1,-\zeta\mathbf1)\,$ in\/ $\,\bbE\hn\times\nh\bbE\,$ has the 
domain\/ $\,\bbR$, and the metric\/ $\,g\,$ on\/ $\,\rn$ defined by\/ 
{\rm(\ref{gjk})} -- {\rm(\ref{tug})} is complete.
\end{theorem}
\begin{proof}The solution 
$\,\bbR\ni t\mapsto\by\nnh_{1,0}\w(t)=-\nh2\hh\tanh\hs nt\,$ (times the 
identity $\,\mathbf1$) of Example~\ref{thsol} leads, via Remark~\ref{group}, 
to further solutions $\,t\mapsto \by\nnh_{a,b}\w(t)=a\by\nnh_{1,0}\w(at+b)$, 
where $\,a,b\in\bbR\,$ and $\,a\ne0$. Suitably chosen and fixed such $\,a,b\,$ 
clearly realize, at $\,t=0$, any prescribed initial data 
$\,(\xi\mathbf1,-\zeta\mathbf1)=(\by\nnh_{a,b}\w(0),\dby\nnh_{a,b}\w(0))
\in\bbR\times(0,\infty)$. Setting 
$\,x\nh_{a,b}\w(t)=(\by\nnh_{a,b}\w(t),\dby\nnh_{a,b}\w(t))\,$ and 
$\,z_\pm=\mp2|a|(\mathbf1,\mathbf0)\,$ we get $\,x\nh_{a,b}\w(t)\to z_\pm$ as 
$\,t\to\pm\infty$. In the discussion preceding (\ref{shl}), applied to 
$\,q=\mp2|a|$, both choices of the sign $\,\pm\,$ lead to the same $\,N\nh$, 
given by $\,\mathrm{s}=-n(n\hn-\nh1)\hh q^2\nh$, and the same $\,w$, while 
$\,z_+\w,z_-\w\in N\hs$ are two different zeros of $\,w$. Using (\ref{shl}) 
we now choose neighborhoods $\,\,U\hskip-2.3pt_\pm\w$ of $\,z_\pm\w$ in 
$\,N\hs$ satisfying the assertion of Lemma~\ref{trapd} for $\,x(t)\,$ equal to 
our $\,x\nh_{a,b}\w(t)$, and $\,t'_\pm\in\bbR\,$ with 
$\,x\nh_{a,b}\w(t'_\pm)\in\,U\hskip-2.3pt_\pm\w$. Since 
$\,z_\pm=\mp2|a|(\mathbf1,\mathbf0)$, we may also require that
\begin{equation}\label{ynz}
\mp y\nh_j\w>|a|\,\mathrm{\ whenever\ 
}\,(\nh y_2\w,\dots,y_n\w,p_2\w,\dots,p_n\w)\in\,U\hskip-2.3pt_\pm\w\mathrm{\ 
and\ }\,j\in\{2,\dots,n\}\hh.
\end{equation}
By continuity, $\,x(t'_\pm)\in\,U\hskip-2.3pt_\pm\w$ for some neighborhood 
$\,\,U\hskip-1.7pt_0\w$ of $\,x\nh_{a,b}\w(0)\,$ in $\,N\hs$ and all integral 
curves $\,t\mapsto x(t)\in N\hs$ of $\,w\,$ with 
$\,x(0)\in\,U\hskip-1.7pt_0\w$. The image of 
$\,(0,\infty)\times U\hskip-1.7pt_0\w$ under the dif\-feo\-mor\-phism $\,F\hs$ 
of Remark~\ref{mltdf} is now a neighborhood of 
$\,x\nh_{a,b}\w(0)=(\xi\mathbf1,-\zeta\mathbf1)$ in $\,\bbE\hn\times\nh\bbE$, 
the existence of which constitutes our assertion: according to 
Remark~\ref{mltdf}, this $\,F\nh$-im\-age equals the union of 
$\,F\hskip-3pt_c\w(U\hskip-1.7pt_0\w)\,$ over $\,c>0$, and each 
$\,F\hskip-3pt_c\w$ maps $\,N\hs$ dif\-feo\-mor\-phic\-al\-ly onto the 
$\,\mathrm{s}\hh$-pre\-im\-age of the value $\,-n(n\hn-\nh1)\hh(cq)^2\nh$, 
while the push-for\-ward, under 
$\,F\hskip-3pt_c\w:N\nh\to F\hskip-3pt_c\w(N)$, of $\,w\,$ obtained by 
restricting $\,v\,$ to $\,N\nh$, is the restriction of $\,v/c\,$ to 
$\,F\hskip-3pt_c\w(N)$. However, the discussion preceding (\ref{shl}), and 
(\ref{shl}) itself, apply to every $\,q\ne0$, and the use of $\,v/c\,$ rather 
than $\,v\,$ makes no difference (Remark~\ref{mltpl}). Now (\ref{ynz}) 
combined with Remark~\ref{compl} yields completeness of $\,g$.
\end{proof}
Our next result shows that the examples arising from Lemma~\ref{xmple}(e) are 
not generally Ric\-ci-par\-al\-lel, or locally reducible, or (when $\,n\ge4$) 
con\-for\-mal\-ly flat.
\begin{theorem}\label{cpopn}The lo\-cal-isom\-e\-try types of Riemannian\/ 
$\,n$-man\-i\-folds satisfying\/ {\rm(\ref{dvr})} -- {\rm(\ref{gnr})} form a 
set with a nonempty interior in the\/ $\,(2n-3)$-di\-men\-sion\-al moduli 
space of Remark\/~{\rm\ref{modsp}}.
\end{theorem}
\begin{proof}
According to Theorem~\ref{lcstr}, the lo\-cal-i\-som\-e\-try types of 
all $\,n$-di\-men\-sion\-al $\,(M\nh,g)\,$ with (\ref{dvr}) -- (\ref{gno}) 
arise from (\ref{gjk}) when one chooses a maximal Ric\-ci-ge\-ner\-ic solution 
$\,I\ni t\mapsto\by\in\bbE\,$ of (\ref{hcu}), and then fixes a smooth curve 
$\,I\ni t\mapsto(g_{22}\w(t),\dots,g_{nn}\w(t))\in(0,\infty)^{n\hn-\nh1}\,$ 
satisfying (\ref{tug}). Restricting our discussion to the case where 
$\,0\in I\nh$, and then pa\-ram\-e\-triz\-ing such solutions (allowed, this 
time, not to be Ric\-ci-ge\-ner\-ic) by their initial data at $\,t=0$, we 
identify them with points of a specific Euclidean space, and completeness of 
$\,g\,$ is guaranteed by Theorem~\ref{cplex} once one assumes (as we do from 
now on) that the initial data range over a certain nonempty open subset of 
the latter space. Now, as in Remark~\ref{modsp}, if $\,n\ge3$, we can make the 
Ric\-ci eigen\-val\-ue functions $\,\mu_1\w(0),\dots,\mu_n\w(0)\,$ of 
Lemma~\ref{xmple} mutually distinct (which leads to 
Ric\-ci-ge\-ner\-ic\-i\-ty) just by ensuring that 
$\,(\mu_2\w(0),\dots,\mu_n\w(0))$ does not lie within a specific finite 
union of hyper\-planes in $\,\bbE$. However, rather than using any prescribed 
$\,\by(0)$, cf. Remark~\ref{modsp}, let us require 
$\,y_1\w(0),\dots,y_n\w(0)\,$ to be all nonzero. 
This amounts to imposing on the solution $\,t\mapsto\by\,$ of (\ref{hcu}) a 
{\it further open condition\/} implying (see the proof of 
Lemma~\ref{xmple}) that $\,R\hn_{j1,\hs j}\w(0)\ne0$, and so $\,g$ is not 
Ric\-ci-par\-al\-lel. In the proof of Lemma~\ref{xmple} we also saw that 
$\,\vg_{\hskip-2.7ptj\hn j}^{\hs1}(0)\ne0$ and, consequently, $\,g\,$ cannot 
be locally reducible. (If it were, the Ric\-ci eigen\-vec\-tor fields 
$\,\partial\nh_1\w,\dots,\partial\nh_n\w$ of Lemma~\ref{xmple}, with distinct 
eigen\-val\-ue functions $\,\mu_1\w,\dots,\mu_n\w$, would each be tangent to 
one or the other parallel factor distribution, giving 
$\vg_{\hskip-2.7ptj\hn j}^{\hs1}\nh=0\hs$ with some $\,j=2,\dots,n$.) 
For $\,k\ne j$, one easily verifies that 
$\,g^{j\hn j}g^{kk}\nnh R\hn_{jkjk}\w=-y\nh_j\w y\nh_k\w$. Therefore, if 
$\,W\nnh$ denotes the Weyl tensor, 
$\,(n\hn-\nh1)(n\hn-\nh2)g^{j\hn j}g^{kk}W\hskip-3.4pt_{jkjk}\w\hs
=\,\,2\hskip1.7pt\mathrm{tr}\hskip1.7pt\dby\,-\,\mathrm{tr}\hskip1.7pt\by^2\hs
-\,(\mathrm{tr}\hskip1.7pt\by)^2
+\,(n\hn-\nh1)[(y\nh_j\w+y\nh_k\w)\hs\mathrm{tr}\hskip1.7pt\by
-(n\hn-\nh2)y\nh_j\w y\nh_k\w-\dot y\nh_j\w-\dot y\nh_k\w]$, where 
$\,\dot y\nh_j\w$ appears with the coefficient $\,3-n$. An enhanced version 
of the last open condition thus precludes con\-for\-mal flat\-ness of our 
examples when $\,n\ge4$.
\end{proof}

\setcounter{section}{1}
\renewcommand{\thesection}{\Alph{section}}
\setcounter{theorem}{0}
\renewcommand{\thetheorem}{\thesection.\arabic{theorem}}
\section*{Appendix: Warped products with harmonic curvature}\label{wp}
\setcounter{equation}{0}
For the reader's convenience, we gather here some facts that are well known 
\cite{kim-cho-hwang} and easily verified. The repeated indices are always 
summed over. In (\ref{war}) we set $\,m=\dim\bM\,$ and $\,\p=\dim\varSigma$, 
assuming that $\,m\p\ge1\,$ and $\,\phi:\bM\to(0,\infty)\,$ is nonconstant. 
Thus, $\,\dim M=n\,$ with $\,n=m+\p\,\ge\hs2$. We use product coordinates 
$\,x^\lambda$ in $\,M\nh$, consisting of local coordinates $\,x^i$ for 
$\,\bM\,$ and $\,x^a$ for $\,\varSigma$, declaring
\begin{equation}\label{rin}
\lambda,\mu,\nu\in\{1,\dots,n\}\hh,\hskip5pti,j,k
\in\{1,\dots,m\}\hh,\hskip5pta,b,c\in\{m+1,\dots,n\}
\end{equation}
to be our index ranges. Therefore, $\,\bg_{i\hn j}\w$ as well as 
$\,\theta\nh=\nh\log\phi\,$ depend only on the variables 
$\,x^k\nh$,\and $\,\h_{ab}\w$ only on $\,x^c\nh$, that is,
%\begin{equation}\label{gij}
$\,\partial\hn_a\w\hs\bg\hn_{i\hn j}\w\hs=\,\partial\hn_a\w\theta\hs
=\,\partial_i\w \h_{ab}\w\hs=\,0$. Furthermore,
\begin{equation}\label{gij}
g_{i\hn j}\w\nh=\bg_{i\hn j}\w\hh,\hskip6ptg_{ia}\w\nh=g_{ai}\w\hs
=\,0\hh,\hskip6ptg_{ab}\w\nh=e^{2\theta}\nh \h_{ab}\w\hh.
\end{equation}
For the Chris\-tof\-fel symbols 
$\,\vg_{\hskip-2.7pt\lambda\mu}^{\hs\nu},\,\bvg_{\hskip-2.7ptij}^{\hs k},
\,H\nnh_{ab}^{\hs c}$ of $\,g,\bg,\eta$, their Ric\-ci-ten\-sor components 
$\,R_{\lambda\mu}\w,\,\br_{i\hn j}\w,\,P\hskip-2.7pt_{ab}\w$, and the 
components 
$\,\bna_{\hskip-2.7pti}\w\hskip-2.2pt\bna_{\hskip-2.7ptj}\w\nh\theta\,$ of 
the $\,\bg\hh$-Hess\-i\-an of $\,\theta$, one has
\[
\begin{array}{l}
g^{\hs ij}\nnh=\bg^{\hh ij},\hskip5ptg^{\hs ia}\nnh=g^{\hs ai}\nh
=0\hh,\hskip5ptg^{\hs ab}\nnh=e^{-\nh2\theta}\nh \h^{ab},\hskip5pt
\vg_{\hskip-2.1ptij}^{\hs k}\nh
=\bvg_{\hskip-3.7ptij}^{\hs k}\hh,\hskip5pt
\vg_{\hskip-2.1ptia}^{\hs k}\nh=\vg_{\hskip-2.1ptij}^{\hs a}\nh=0\hh,\hskip5pt
\vg_{\hskip-2.1ptia}^{\hs b}\nh=\delta_{\nh a}^b\theta\nnh_{,\hh i}\w\hh,\\
\vg_{\hskip-3.2ptab}^{\hs i}
=-e^{2\theta}\nnh\h_{ab}\w\hs\theta^{,i},\hskip8pt
\vg_{\hskip-2.7ptab}^{\hs c}=H\nnh_{ab}^{\hs c}\hh,\hskip8pt
R_{i\hn j}\w\nnh=\nh\br_{i\hn j}\w\nh
-\p\hh[\hs\bna_{\hskip-2.7pti}\w\hskip-2.2pt\bna_{\hskip-2.7ptj}\w\nh\theta
+\nh\theta\nnh_{,\hh i}\w\hh\theta\nnh_{,\hh j}\w]\hh,
\end{array}
\]
while, in terms of the $\,\bg$-La\-plac\-i\-an 
$\,\hs\overline{\nh\Delta\nh}\hh$,
\begin{equation}\label{ria}
R_{ia}\w
=\,0\hh,\hskip16ptR_{ab}\w=\,P\hskip-2.7pt_{ab}\w\nnh
-\nh \p^{-\nh1}\nh e^{(2-\p)\theta}
[\hs\overline{\nh\Delta\hh}e^{\p\theta}]\hh\h_{ab}\w\hh.
\end{equation}
The components 
$\,R\nh_{\lambda\mu\hn,\hs\nu}\w,
\bna_{\hskip-2.7pti}\w\hskip-2.2pt\br\nh_{jk}\w,
\hs D\nnh_c\w P\hskip-2.7pt_{ab}\w$ of the covariant derivatives of the 
Ric\-ci tensors of $\,g,\bg,\eta\,$ satisfy, with the usual conventions 
$\,\theta\nnh_{,\hh i}\w\nh=\partial\nh_i\w\theta\,$ and 
$\,\theta\hh^{,\hh i}\nh=\bg^{i\hn j}\partial\nh_j\w\theta$, the relations
\begin{equation}\label{bnr}
\begin{array}{l}
R\nh_{j\hn k,\hs i}\w=\bna_{\hskip-2.7pti}\w\hskip-2.2pt\br\nh_{j\hn k}\w\nh
-\p\hh[\bna_{\hskip-2.7pti}\w\hskip-2.2pt\bna_{\hskip-2.7ptj}\w\hskip-2.2pt
\bna_{\hskip-2.7ptk}\w\nh\theta+\,\bna_{\hskip-2.7pti}\w\hskip-3pt
(\nh\theta\nnh_{,\hh j}\w\hh\theta\nnh_{,\hh k}\w)]\hh,
\hskip28ptR_{i\hn j,\hs a}\w\nh=R_{a\hn j,\hs i}\w\nh=0\hh,\\
R_{ib,\hs a}\w\nh
=e^{2\theta}\hn(\p^{-\nh1}\nh e^{-\p\theta}\nh
[\hs\overline{\nh\Delta\hh}e^{\p\theta}]\hs\theta\nnh_{,\hh i}\w\nh
+\nh[\br_{i\hn j}\w\nh
-\p\bna_{\hskip-2.7pti}\w\hskip-2.2pt\bna_{\hskip-2.7ptj}\w\nh\theta
-p\hskip.85pt\theta\nnh_{,\hh i}\w\hh\theta\nnh_{,\hh j}\w]\hs\theta\hh^{,\hs j})\hh 
\h_{ab}\w\nnh-\nh\theta\nnh_{,\hh i}\w\hh P\hskip-2.7pt_{ab}\w\hh,\\
R_{ab,\hs i}\w\nh
=-\nh \p^{-\nh1}\nh e^{2\theta}\nh
(e^{-\p\theta}\overline{\nh\Delta\hh}e^{\p\theta})\nnh_{,\hh i}\w\hh\h_{ab}\w\nnh
-\nh2\hh\theta\nnh_{,\hh i}\w\hh P\hskip-2.7pt_{ab}\w\hh,\hskip16pt
R_{ab,\hs c}\w=D\nnh_c\w P\hskip-2.7pt_{ab}\w\hh.
\end{array}
\end{equation}
Let (a) -- (e) refer to parts of Lemma~\ref{warhc}, which we now proceed 
to prove. First,
\begin{enumerate}
  \def\theenumi{{\rm\roman{enumi}}}
\item [(f)] $R_{ab,\hs i}\w\nh=R_{ib,\hs a}\w\,\hs$ for all $\,\,i,a,b\,\,$ as 
in (\ref{rin}) if and only if one has (a) and (e).
\end{enumerate}
In fact, it suffices to verify (f) on the dense set 
$\,(U\hskip-1pt\cup U')\times\hn\varSigma\,\hs\subseteq\hs M\nh$, for the 
interior $\,\,U$ of the zero set of $\,d\hh\theta\,$ in $\,\bM\,$ and the 
subset $\,\,U'$ on which $\,d\hh\theta\ne0$. On $\,\,U\nnh$, according to 
(\ref{bnr}), $\,R_{ab,\hs i}\w\nh=0=R_{ib,\hs a}\w$ since 
$\,\overline{\nh\Delta\hh}e^{\p\theta}\nh=0$. Similarly, on $\,\,U'\nh$, 
the equality $\,R_{ab,\hs i}\w\nh=R_{ib,\hs a}\w$ amounts, by (\ref{bnr}), to 
the condition $\,P\hskip-2.7pt_{ab}\w\nh=\kappa\hh\h_{ab}\w$, for a function 
$\,\kappa\,$ on $\,\varSigma$ which must be constant, as it depends only on 
the variables $\,x^{\hs j}$ that are local coordinates in $\,\bM\nh$. Formulae 
(\ref{bnr}) also show that $\,\kappa\,$ is characterized by the relation  
$\,-\kappa\hh e^{-\nh2\theta}d\hh\theta
=\p^{-\nh1}(\hn d\hh[e^{-\p\theta}\overline{\nh\Delta\hh}e^{\p\theta}]+
e^{-\p\theta}\hn[\hs\overline{\nh\Delta\hh}e^{\p\theta}]\,d\hh\theta)
+\overline{\mathrm{r}}(\bna\nh\theta,\,\cdot\,)
-\p\hh\bg(\bna\nh\theta,\nnh\bna\nh\theta)\hh d\hh\theta
-\p\hs d[\bg(\bna\nh\theta,\nnh\bna\nh\theta)]/2\,$ which, rewritten in terms 
of $\,\phi=e^\theta\nh$, becomes (e).

The equivalence of (e) and (c) is in turn obvious from (\ref{rcd}). Next, by 
(\ref{bnr}),
\begin{enumerate}
  \def\theenumi{{\rm\roman{enumi}}}
\item [(g)] $R\nh_{j\hn k,\hs i}\w\nh
=R\nh_{ik,\hs j}\w\,\hs$ for all $\,\,i,j,k\,\,$ 
with (\ref{rin}) if and only if (b) holds,
\end{enumerate}
since $\,\phi^{-\nnh1}\bna\nh d\phi=\bna\nh d\theta
+d\theta\otimes\hs d\theta$. The main claim of Lemma~\ref{warhc} is thus 
immediate: harmonicity of the curvature amounts to the Co\-daz\-zi equation 
for the Ric\-ci tensor, cf.\ (\ref{cod}.i), while (\ref{bnr}) clearly reduces 
the latter to the cases (f) -- (g).

Finally, (d) follows from (a) and (\ref{ria}).

\bibliographystyle{amsplain}

\end{document}